\documentclass[letterpaper]{amsart}

\usepackage[utf8]{inputenc}
\usepackage{amsmath}
\usepackage{amssymb}
\usepackage{amsthm}
\usepackage{tikz-cd}
\usepackage{mathrsfs}
\usepackage{enumitem}
\usepackage[alphabetic]{amsrefs}
\usepackage{tikz}
\usepackage{float}
\usepackage{hyperref}
\usepackage{multirow}

\theoremstyle{plain}
\newtheorem{teo}{Theorem}[section]
\newtheorem{prop}[teo]{Proposition}
\newtheorem{lemma}[teo]{Lemma}
\newtheorem{cor}[teo]{Corollary}
\theoremstyle{definition}

\newtheorem{defin}[teo]{Definition}

\theoremstyle{remark}
\newtheorem{remark}[teo]{Remark}
\newtheorem{claim}[teo]{Claim}
\numberwithin{equation}{section}

\newcommand{\A}{\mathcal{A}}
\newcommand{\afrak}{\mathfrak{a}}
\newcommand{\C}{\mathbb{C}}
\newcommand{\D}{\mathcal{D}}
\newcommand{\dbf}{\mathbf{d}}
\newcommand{\DC}{\mathrm{D}}
\newcommand{\Ftors}{\mathscr{F}}
\newcommand{\K}{\mathcal{K}}
\newcommand{\nbf}{\mathbf{n}}
\renewcommand{\O}{\mathscr{O}}
\renewcommand{\P}{\mathbb{P}}
\newcommand{\Q}{\mathbb{Q}}
\newcommand{\R}{\mathbb{R}}
\newcommand{\rbf}{\mathbf{r}}
\newcommand{\T}{\mathcal{T}}
\newcommand{\Ttors}{\mathscr{T}}
\newcommand{\vbf}{\mathbf{v}}
\newcommand{\Z}{\mathbb{Z}}

\let\Im\relax
\let\Re\relax
\DeclareMathOperator{\Bl}{Bl}
\DeclareMathOperator{\Coh}{Coh}
\DeclareMathOperator{\coker}{coker}
\DeclareMathOperator{\Ext}{Ext}
\DeclareMathOperator{\Hom}{Hom}
\DeclareMathOperator{\Im}{Im} 
\DeclareMathOperator{\Ku}{Ku}
\DeclareMathOperator{\Pic}{Pic}
\DeclareMathOperator{\Re}{Re}
\DeclareMathOperator{\rk}{rk}
\DeclareMathOperator{\Stab}{Stab}
\DeclareMathOperator{\Sym}{Sym}

\newcommand{\ev}{\mathrm{ev}}
\newcommand{\GL}{\mathrm{GL}}
\newcommand{\id}{\mathrm{id}}
\newcommand{\num}{\mathrm{num}}
\newcommand{\Vect}{\mathrm{Vect}}

\newcommand{\abs}[1]{\left\lvert#1\right\rvert}

\begin{document}

\title{Weak stability conditions on coherent systems of genus four curves}
\author{Nicolás Vilches}
\address{Department of Mathematics, Columbia University, 2990 Broadway, New York, NY 10027, USA}
\email{nivilches@math.columbia.edu}
\begin{abstract}
The derived category of coherent systems is an interesting triangulated category associated with a smooth, projective curve $C$. These categories admit Bridgeland stability conditions, as recently shown by Feyzbakhsh and Novik. Their construction depends explicitly on the higher rank Brill--Noether theory of $C$.

In this short note, we study the Feyzbakhsh--Novik stability conditions for a general curve of genus four. We show that these stability conditions degenerate to a stability condition on the Kuznetsov component of the corresponding nodal cubic threefold, using a result of Alexeev--Kuznetsov.
\end{abstract}
\maketitle

\section{Introduction}

The study of coherent systems on curves has been an active area of research in the last thirty years. A \emph{coherent system} on a smooth, projective curve $C$ consists of a pair $(V, E)$, where $E$ is a vector bundle on $C$ and $V \subset H^0(C, E)$. In a similar way to classical slope-stability, there is a notion of \emph{$\alpha$-stability} on coherent systems, depending on an extra parameter $\alpha$. We refer to \cite{FN25}*{\textsection 1.1} and the references therein for an historic summary. 

We are interested of exploring the derived category of coherent systems, following the foundational works \citelist{\cite{AK25} \cite{FN25}}. The first observation is that the category of coherent systems is not abelian; instead, one should relax the condition $V \subset H^0(C, E)$ to a morphism $V \to H^0(C, E)$ (possibly zero). These generalized coherent systems form an abelian category $\T_C$, which gives us a derived category $\DC^b(\T_C)$

The main result of \cite{FN25} ensures that the category $\DC^b(\T_C)$ admits \emph{Bridgeland stability conditions} (cf. \cite{JRLVM25}). One of the constructions of \cite{FN25} is given by the theorem below; see also our discussion in Section \ref{sec:stab}.

\begin{teo}[Feyzbakhsh--Novik, \cite{FN25}*{\textsection 3}] \label{teo:intro_FN}
Let
\[ \Phi_C(x) = \lim_{\varepsilon \to 0}\sup\{ h^0(C, E)/\rk(E): E\text{ semistable, } \mu(E) \in (x-\varepsilon, x+\varepsilon)\}. \]
Then there exists a family of stability conditions on $\DC^b(\T_C)$, depending of two parameters $b, w$ satisfying $w>\Phi_C(b)$.
\end{teo}

We point out that the function $\Phi_C$ involves the higher rank Brill--Noether theory of $C$. As the next result highlights, the structure of $\DC^b(\T_C)$ is heavily controlled by the Brill--Noether theory of $C$. We refer to the discussion of these two results in Sections \ref{sec:coh} and \ref{sec:genfour} for more details.

\begin{teo}[Alexeev--Kuznetsov] \label{teo:intro_AK}
\begin{enumerate}
\item (\cite{AK25}*{Proposition 3.12}) If $L$ is a \emph{Brill--Noether--Petri extremal} line bundle on $C$, then there is an associated exceptional object $\afrak(L) \in \DC^b(\T_C)$.
\item (\cite{AK25}*{Corollary A.2}) Let $Y$ be a nodal cubic threefold, and let $C$ be the genus 4 curve parametrizing lines passing through the node of $Y$. Then, the Kuznetsov component $\Ku(Y)$ of $Y$ is equivalent to the Verdier quotient of $\DC^b(\T_C)$ by the two objects $\afrak(L_1), \afrak(L_2)$ associated to the two trigonal line bundles on $C$. 
\end{enumerate}
\end{teo}

Our main result connects these two theorems as a degeneration of Bridgeland stability conditions.

\begin{teo} \label{teo:intro_stab}
Let $Y$ be a nodal cubic threefold, and let $C$ be the associated genus four curve. Then, the Feyzbakhsh--Novik stability conditions $\{\sigma_{b, w}\}$ of Theorem \ref{teo:intro_FN} degenerate to a weak stability condition $\sigma_{3,2}$, induced by a stability condition $\overline{\sigma}_{3, 2}$ on $\Ku(Y)$.
\end{teo}

This is motivated by the general philosophy that certain degenerations of Bridgeland stability conditions on a triangulated category $\D$ should converge to stability conditions on appropriate Verdier quotients of $\D$. See \citelist{\cite{BPPW22} \cite{Bol23} \cite{HLR25}} for various incarnations of this principle, and \citelist{\cite{Cho24} \cite{Cho25} \cite{Vil25}} for some recent examples. 

As a sample application, we will prove in Section \ref{sec:moduli} the following consequence.

\begin{cor} \label{cor:intro_moduli}
Let $C$ be a general genus four curve. There exists a numerical vector $\vbf \in K^{\num}(\T_C)$ satisfying the following. 
\begin{enumerate}
\item For each $w>2$, the moduli space $M_w$ of $\sigma_{3, w}$-semistable objects has an connected component $M_w^\circ$ isomorphic to $\Sym^2(C)$. 

\item Given $p+q, p'+q' \in \Sym^2(C)$, we have that the two corresponding objects in $M_w^\circ$ become S-equivalent in $\sigma_{3, 2}$ if and only if the two lines determined by $p+q, p'+q'$ in $\P^3$ (under the canonical embedding) intersect at a point $r \in C$. 
\end{enumerate}
\end{cor}

This result should be compared to the normalization of the Fano variety of lines on a nodal cubic threefold; see the end of Section \ref{sec:moduli} for more details.

\subsection{Structure of the paper}

We will start by reviewing the construction of the derived category of coherent systems in Section \ref{sec:coh}. We will then discuss the \emph{Brill--Noether} function of a curve $C$ in Section \ref{sec:bn}, and will use it to construct the stability conditions of Theorem \ref{teo:intro_FN} in Section \ref{sec:stab}.

We will then specialize to general genus four curves. Their geometry will be reviewed in Section \ref{sec:genfour}, and their relation to nodal cubic fourfolds (including Theorem \ref{teo:intro_AK}) will be discussed in Section \ref{sec:sod}.

The next two sections are devoted to the proofs of our main results. We will prove Theorem \ref{teo:intro_stab} in Section \ref{sec:pfmain}, and we will prove Corollary \ref{cor:intro_moduli} in Section \ref{sec:moduli}. Finally, we give a short list of future directions in Section \ref{sec:questions}.

\subsection{Acknowledgements}

I am deeply grateful to my PhD advisor, Giulia Saccà, for her constant support and encouragement through the years. I am thankful to Laura Pertusi and Alekos Robotis for discussions related to this project, and to Alexander Kuznetsov for comments on a preliminary version of this paper.

This work was partially supported by the National Science Foundation (grant number DMS-2052934) and by the Simons Foundation (grant number SFI-MPS-MOV-00006719-09).

\section{Coherent systems} \label{sec:coh}

In this section we briefly recall the construction of the category of \emph{(generalized) coherent systems} on a curve, following \cite{FN25}*{\textsection 2}. See also \cite{AK25}*{\textsection 3} for a slightly different approach. We will then discuss some properties of its derived category, including the \emph{Brill--Noether exceptional objects} from \cite{AK25}*{\textsection 3.2.}

Through this section we fix a smooth, projective curve $C$ of genus $g$. We start by recalling the following construction.

\begin{defin}[\cite{FN25}*{p. 5}]
We let $\T_C$ be the category of triples\footnote{Note that in \cite{FN25} the notation $[V \otimes \O_C \xrightarrow{\phi} E]$ is also used.} $(V, E; \phi)$, where $V$ is a vector space, $E \in \Coh(C)$, and $\phi\colon \O_C \otimes V \to E$. Morphisms from $(V, E; \phi)$ to $(V', E'; \phi')$ consists of compatible maps $V \to V'$, $E \to E'$.
\end{defin}

It is not difficult to see that $\T_C$ is an abelian category. This way, we let $\DC^b(\T_C)$ be its bounded derived category.

\begin{lemma}[\cite{FN25}*{\textsection 2.1}] \label{lemma:aug_functors}
\begin{enumerate}
\item There exists an adjoint pair $i^\ast \dashv i_\ast$ between $\T_C$ and $\Vect_\C$, with $i^\ast(V, E; \phi)=V$ and $i_\ast(V)=(V, 0; 0)$. The functors are exact, and the corresponding extensions to the derived categories are also adjoint.

\item There exists an adjoint pair\footnote{Note that in \cite{FN25} the notation $j^\dagger$ is used instead.} $j_\ast \dashv j^!$ between $\Coh(C)$ and $\T_C$ with $j_\ast(E)=(0, E; 0)$ and $j^!(V, E; \phi) = E$. The functors are exact, and the corresponding extensions to the derived categories are also adjoint.

\item There is a semi-orthogonal decomposition
\[ \DC^b(\T_C) = \langle i_\ast \DC^b(\Vect_\C), j_\ast \DC^b(C) \rangle. \]
Up to a shift, the associated gluing functor is given by
\[ \DC^b(\Vect_\C)^{op} \times \DC^b(C) \to \DC^b(\C), \quad (V^\bullet, F^\bullet) \mapsto R\Hom(V^\bullet \otimes \O_C, F^\bullet). \]
\end{enumerate}
\end{lemma}

It turns out that $i_\ast$ also has a \emph{right} adjoint $i^!\colon \T_C \to \Vect_\C$; and $j_\ast$ has a \emph{left} adjoint $j^\ast\colon \T_C \to \Coh(C)$. We will not be using them explicitly; thus, we refer to \cite{FN25}*{\textsection 2.1} for details.

\begin{claim}
The category $\DC^b(\T_C)$ is equivalent to an admissible
subcategory of $\DC^b(X)$, for a smooth, proper variety $X$. See \cite{AK25}*{Lemma 3.4} for an explicit construction; alternatively, this follows from \cite{Orl16}*{Theorem 4.15}.
\end{claim}

\begin{lemma}[\cite{FN25}*{Proposition 2.2}] \label{lemma:aug_Exts}
Given two objects $T_i = (V_i, E_i; \phi_i) \in \T_C$, we have that $\Ext^k(T_1, T_2)=0$ for $k \neq 0, 1, 2$. Moreover, we have a long exact sequence
\begin{gather*}
0 \to \Hom(T_1, T_2) \to \Hom(V_1, V_2) \oplus \Hom(E_1, E_2) \to \Hom(\O_C \otimes V_1, E_2) \\
\to \Ext^1(T_1, T_2) \to \Ext^1(E_1, E_2) \to \Ext^1(\O_C \otimes V_1, E_2) \to \Ext^2(T_1, T_2) \to 0.
\end{gather*}
\end{lemma}

\begin{cor}[cf. \cite{FN25}*{p. 7}]
The Euler pairing $\chi(T_1, T_2)$ on $\DC^b(\T_C)$ is given by the formula $\chi(T_1, T_2) = \vbf(T_1) M \vbf(T_2)^\top$. Here $\vbf(T) = (\rbf(T), \dbf(T), \nbf(T))=(\rk E, \deg E, \dim V)$ for $T=(V, E; \phi)$, and extended linearly; and 
\[ M = \begin{pmatrix} 1-g & 1 & 0 \\ -1 & 0 & 0 \\ g-1 & -1 & 1 \end{pmatrix}. \]
In particular, we have that the numerical K-theory of $\DC^b(\T_C)$ is isomorphic to $\Z^3$, induced by $\vbf$. 
\end{cor}

We finish our summary by recalling the construction of \emph{BN exceptional objects}, following \cite{AK25}*{\textsection 3.2}. To do so, we recall that a line bundle $L$ on $C$ is \emph{Brill--Noether--Petri extremal} (or \emph{BNP extremal} for short) if the map
\[ H^0(C, L) \otimes H^0(C, L^\vee \otimes \omega_C) \to H^0(C, \omega_C) \]
is an isomorphism. Given such a line bundle, we associate the object
\[ \afrak(L) = (H^0(C, L), L; \ev_L) \in \T_C, \]
where $\ev_L$ is the evaluation map. We call $\afrak(L)$ the \emph{Brill--Noether exceptional object} associated by $L$, or \emph{BN exceptional} for short. The name is justified by the following proposition.

\begin{prop}[\cite{AK25}*{3.12}]
Given a BNP extremal line bundle $L$, the Brill--Noether modification $\afrak(L)$ is an exceptional object in $\DC^b(\T_C)$.
\end{prop}

\section{Brill--Noether function} \label{sec:bn}

In this section we briefly define the \emph{Brill--Noether function} of a curve $C$, following \cite{FN25}*{p. 13}. Informally, this function is introduced to bound the dimension of the space of global sections of semistable sheaves. We let $C$ be a smooth, projective curve of genus $g \geq 1$. 

We start by defining the function $\varphi=\varphi_C\colon \Q \to \R$ via
\[ \varphi_C(\mu) = \sup\left\{ \frac{h^0(C, F)}{\rk(E)}: E \in \Coh(C), \text{slope-semistable, } \mu(E)=\mu \right\}, \]
where $\mu(E)=\deg(E)/\rk(E)$. Note that our assumption on the genus guarantees that for every rational number $\mu$ there is a semistable sheaf with slope $\mu$, hence the supremum is over a non-empty set. The fact that $\varphi(\mu)\neq +\infty$ is a consequence of the following result.

\begin{lemma}[cf. \cite{FN25}*{p. 13}] \label{lemma:bn_prephi}
Let $\mu$ be a rational number.
\begin{enumerate}
\item If $\mu<0$, then $\varphi(\mu)=0$. 
\item We have $\varphi(\mu) = \varphi(2g-2-\mu) + \mu + (1-g)$.
\item If $\mu>2g-2$, then $\varphi(\mu)=\mu + 1-g$.
\item If $0 \leq \mu \leq 2g-2$, then $\varphi(\mu) \leq \mu/2+1$. 
\end{enumerate}
\end{lemma}

\begin{proof}
Part (1) follows from the fact that a non-zero element of $h^0(C, E)$ gives an injection $\O_C \to E$ if $E$ is a vector bundle. Part (2) is a consequence of Riemann--Roch, together that $E^\vee \otimes \omega_C$ is slope-semistable if $E$ is slope-semistable. Part (3) follows from (1) and (2). Finally, part (4) is a consequence of the higher rank Clifford's theorem, cf. \cite{BPGN97}*{Theorem 2.1}.
\end{proof}

It turns out that for our purposes it is better to write a slightly different function. We set $\Phi = \Phi_C\colon \R \to \R$ via
\[ \Phi_C(x) = \limsup_{\varepsilon \to 0} \{ \varphi(\mu): \mu \in (x-\epsilon, x+\epsilon) \cap \Q \}. \]
The following lemma follows immediately from Lemma \ref{lemma:bn_prephi}.

\begin{lemma}[\cite{FN25}*{3.2}]
\begin{enumerate}
\item The function $\Phi_C$ satisfies the analogous properties of Lemma \ref{lemma:bn_prephi}. 
\item We have that $\Phi_C$ is upper-semicontinuous and $\Phi_C(\mu(E)) \geq h^0(C, E)/\rk(E)$ for any semistable vector bundle $E$ on $C$. 
\end{enumerate}
\end{lemma}

\section{Tilting stability conditions} \label{sec:stab}

In this section we recall a construction of Bridgeland stability conditions on the derived category of coherent systems, following \cite{FN25}*{\textsection 3}. Through this section we let $C$ be a smooth, projective curve of genus $g \geq 1$. 

We start by considering the slope function $\mu$ on $\T_C$ given by
\[ \mu(V, E; \phi) = \begin{cases} \deg E/\rk E & \text{if }\rk E>0, \\ +\infty & \text{otherwise.} \end{cases} \]
An object $T$ in $\T_C$ is \emph{$\mu$-semistable} if $\mu(T') \leq \mu(T)$ for any non-zero subobject $T' \subseteq T$. A standard argument shows that every object in $\T_C$ admits a unique Harder--Narasimhan filtration by $\mu$-semistable objects. 

Fix a real number $b$. We let $\Ttors^b$ be the full subcategory of $\T_C$ generated by $\mu$-semistable objects of slope bigger than $b$; and we let $\Ftors^b$ be the full subcategory of $\T_C$ generated by $\mu$-semistable objects of slope smaller or equal than $b$. We get that $(\Ttors^b, \Ftors^b)$ is a torsion pair, and we denote $\A(b) = \langle \Ttors^b, \Ftors^b[1] \rangle$ the tilt.

\begin{teo}[\cite{FN25}*{Theorem 3.3}] \label{teo:stab_FN}
Given $(b, w) \in \R^2$ satisfying $w>\Phi(b)$, we have that $\sigma_{b, w} =(\A(b), Z_{b, w})$ is a Bridgeland stability condition, where
\[ Z_{b, w}\colon K^{\num}(\T_C) \to \C, \qquad Z_{b, w}(T) = (-\nbf(T)+w \rbf(T)) + i (\dbf(T)-b\rbf(T)) \]
is the central charge. This defines a two-dimensional continuous family of stability conditions on $\DC^b(\T_C)$.
\end{teo}

We will not give a proof of this result, and refer to \cite{FN25}*{pp. 13--4}. Let us briefly justify the bound $w>\Phi_C(b)$ for the stability conditions. Given $T=(V, E; \phi)$, we have that $\Im(T) = \rk(E) \cdot (\mu(E)-b)$. Thus, if $E$ is slope-semistable of slope $>b$ (resp. $\leq b$), then $\Im(T)>0$ (resp. $\Im(T[1]) \geq 0$).

Assume now that $b$ is rational. Given a semistable vector bundle $E$ of slope $b$, then the object $(H^0(C, E), E; \ev_E)[1] \in \A(b)$ has $Z_{b, w}(E) = h^0(C, E)-w \rk(E)$. We thus need $w > \varphi_C(b)$ to guarantee that $Z_{b, w}$ is a central charge on $\A(b)$. A deformation argument forces us to get $w>\Phi_C(b)$.

\begin{lemma}[cf. \cite{FN25}*{Lemma 3.7}] \label{lemma:stab_quad}
Let $b, w_0$ be rational numbers satisfying $w_0>\Phi_C(b)$. Fix some $s, t>0$ such that $s(x-b_0)^2 + w_0-t > \Phi_C(x)$ holds for any real number $x \neq b_0$. Set $Q(r, d, n) = s(d-b_0r)^2 + r^2(w_0-t)-nr$. 
\begin{enumerate}
\item For any $w\geq w_0$, the quadratic form $Q$ is negative definite on $\ker Z_{b, w}$.
\item For any $w \geq w_0$, we have that $Q(T) \geq 0$ for any $\sigma_{b, w}$-semistable object $T\in \A(b)$ with $\Im Z_{b, w}(E)>0$. 
\item Assume additionally that $w_0-t>\Phi_C(b)$. Then $\sigma_{b, w}$ satisfies the support property with respect to $Q$, for any $w \geq w_0$. 
\end{enumerate}
\end{lemma}

\begin{proof}
Part (1) is clear, and part (3) is claimed in \cite{FN25}*{Lemma 3.7}. For (2), we follow the proof of \cite{FN25}*{Lemma 3.7}. Assume that $R \neq 0$ is a $\sigma_{b, w}$-semistable object with $\Im Z_{b, w}(T)>0$. If $T$ remains $\sigma_{b, w'}$-semistable for all $w' \geq w$, then \cite{FN25}*{Lemma 3.7} applies. Thus either $T=(V, E; \phi)$ with $E$ a semistable vector bundle; or $H^0(T)=(\C^{r}, 0; 0)$ and $H^{-1}(T)=(V, E; \phi)$ with $E$ a semistable vector bundle. But in both cases the assumption $\Im Z_{b, w}(T)>0$ ensures that $\mu(E) \neq b$. This way, the proof of \cite{FN25}*{Lemma 3.7} applies verbatim.
\end{proof}

\section{Genus four curves} \label{sec:genfour}

In this section we will study the geometry of general genus four curves. We will first recall their canonical embedding, and use it to relate them with nodal cubic threefolds, following the exposition of \citelist{\cite{vdGK10}*{\textsection 2} \cite{AK25}*{App. A}}. We will then describe their (higher rank) Brill--Noether theory, following \cite{LN17}. 

Fix a smooth, projective curve $C$ of genus four. We start by considering the canonical map $\phi=\phi_{\abs{K_C}}\colon C \to \P^3$. We assume that $C$ is non-hyperelliptic, so that $\phi$ is an embedding (e.g. by \cite{Har77}*{Proposition IV.5.2}). We thus identify $C$ with its image inside $\P^3$, which is a $(2, 3)$-complete intersection, by \cite{Har77}*{Example IV.5.2.2}. We fix $Q \subset \P^3$ to be the unique quadric surface containing $C$.

\begin{claim} \label{claim:genfour_quadric}
The quadric $Q$ corresponds to a quadratic form of rank 3 or 4. This follows immediately from the fact that $C$ is not contained in a plane. Moreover, the general curve of genus four is contained in a smooth quadric surface, e.g. by \cite{Har77}*{Exercise IV.5.3}.
\end{claim}

\begin{remark}
For the remainder of this note, we will say that a genus four curve $C$ is \emph{general} if it is non-hyperelliptic, and the associated quadric $Q$ is smooth. 
\end{remark}

Assume that $C$ is a general curve of genus four. The linear system of cubic surfaces in $\P^3$ containing $C$ defines a rational map $\P^3 \dashrightarrow \P^4$.

\begin{claim}[\cite{vdGK10}*{p. 28}] \label{claim:genfour_blowups}
The map $\P^3 \dashrightarrow \P^4$ can be resolved by blowing-up $\P^3$ along $C$. The induced map $X=\Bl_C \P^3 \to \P^4$ factors as $X \to Y \to \P^4$, where $X \to Y$ is the contraction of the strict transform of $Q$ to a point, and $Y \hookrightarrow \P^4$ is a cubic threefold with a single node. 

This process can be reversed. If $Y \subset \P^4$ is a cubic threefold with a single node $p \in Y$, then projection away from $p$ defines a birational map $Y \dashrightarrow \P^3$. This map is resolved by blowing up $p$, to get $X=\Bl_p Y \to \P^3$. The map contracts the strict transform $E$ of the surface spanned by lines passing through $p$.
\end{claim}

The two descriptions of Claim \ref{claim:genfour_blowups} fit in the following diagram.
\begin{equation} \label{eq:genfour_blowups}
\begin{tikzcd}
Q  \arrow[r, hook] \arrow[d] & \Bl_pY \arrow[r, equal, "=:X"'] \arrow[d, "\pi"] & \Bl_C \P^3 \arrow[d, "\alpha"'] & B \arrow[l, hook'] \arrow[d, "\beta"] \\ \{p\} \arrow[r, hook] & Y & \P^3 & C \arrow[l, hook', "i"']
\end{tikzcd}
\end{equation}
We point out that $\Pic(X) \cong \Z^2$, generated by $F$ and $B$, where $F$ is the pullback of an hyperplane in $\P^3$. If $H$ is the pullback of the hyperplane class in $\P^4$, then we have the relations $H=3F-B$ and $Q=2H-B$ in $\Pic(C)$.

We now turn into the Brill--Noether theory of a general genus four curve $C$. As before, we identify $C$ with its image in $\P^3$ via the canonical embedding. We let $Q \subset \P^3$ be the smooth quadric surface containing $C$.

\begin{lemma}
Let $L_1, L_2$ be the two line bundles on $C$ induced by the two rulings of $Q \subset \P^3$. Then $L_1$ and $L_2$ are non-isomorphic line bundles of degree three with $h^0(C, L_i)=2$. Moreover, they are the only Brill--Noether--Petri extremal line bundles of $C$.
\end{lemma}

\begin{proof}
The first part is a direct computation. For the second one, note that a BNP extremal line bundle $L$ must have degree 3, $h^0(C, L)=2$, and be globally generated. Thus, we have that $\abs{L}$ defines a degree 3 morphism into $\P^1$. The result follows now by \cite{Har77}*{Example IV.5.5.2}.
\end{proof}

Recall the definition of the \emph{Brill--Noether function} in Section \ref{sec:bn}. Under our assumptions on $C$, we can give a better bound for $\Phi$.

\begin{prop} \label{prop:genfour_bound}
Let $C$ be a general curve of genus four. We have that
\[ \Phi_C(x) \leq \begin{cases} 1 & x=0, \\ \frac{1}{4}x + \frac{3}{4} & 0<x<2, \\ \frac{1}{3}x + \frac{2}{3} & 2\leq x<\frac{5}{2}, \\ \frac{1}{2}x + \frac{1}{4} & \frac{5}{2} \leq x <3, \\ 2 & x=3. \end{cases} \]
\end{prop}

\begin{proof}
The bound holds for $\varphi_C$ by \cite{LN17}*{Theorem 4.8}. By upper-semicontinuity the same bound holds for $\Phi_C$.
\end{proof}

\begin{cor} \label{cor:genfour_bound}
Let $C$ be a general curve of genus four. We have that $\Phi_C(x) \leq (x-3)^2+1.9$ for all $x \neq 3$.
\end{cor}

\begin{proof}
The bound follows directly by Proposition \ref{prop:genfour_bound} for $0 \leq x <3$. For $3<x\leq 6$ we use that $\Phi_C(x) =\Phi_C(6-x)+x-3$ together with Proposition \ref{prop:genfour_bound}. Finally, for $x<0$ (resp. $x>6$) we use that $\Phi_C(x)=0$ (resp. $\Phi_C(x) =x-3$). We depicted the bound from Proposition \ref{prop:genfour_bound} together with $y=(x-3)^2+1.9$ in Figure \ref{fig:genfour_bound}.
\end{proof}

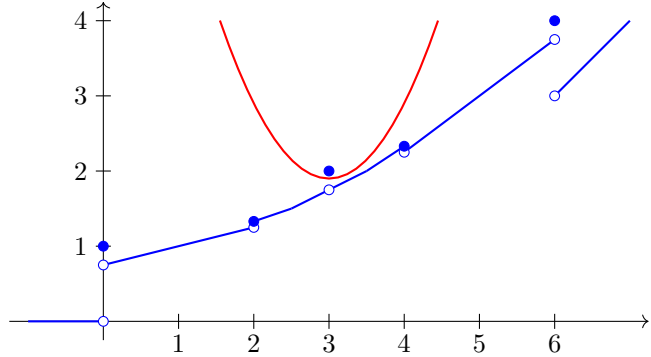
\begin{figure}[htbp]
\centering
\begin{tikzpicture}
\draw[->] (-1.25,0) -- (7.25,0); 
\draw[->] (0,-0.25) -- (0,4.25);

\foreach \y in {1,2,3,4} {\draw (-0.1,\y) -- (0.1,\y); \node at (-0.3,\y) {\y};}
\foreach \x in {1,2,3,4,5,6} {\draw (\x,0.1) -- (\x,-0.1); \node at (\x,-0.3) {\x};}

\draw[thick, blue] (-1,0) -- (0,0);
\draw[thick, blue] (0,0.75) -- (2,1.25);
\draw[thick, blue] (2,1.33) -- (2.5,1.5);
\draw[thick, blue] (2.5,1.5) -- (3.5,2);
\draw[thick, blue] (3.5,2) -- (4,2.33);
\draw[thick, blue] (4,2.25) -- (6,3.75);
\draw[thick, blue] (6,3) -- (7,4);

\draw[fill=white, draw=blue] (0,0) circle [radius=0.065];
\draw[fill=white, draw=blue] (0,0.75) circle [radius=0.065];
\draw[fill=white, draw=blue] (2,1.25) circle [radius=0.065];
\draw[fill=white, draw=blue] (3,1.75) circle [radius=0.065];
\draw[fill=white, draw=blue] (4,2.25) circle [radius=0.065];
\draw[fill=white, draw=blue] (6,3) circle [radius=0.065];
\draw[fill=white, draw=blue] (6,3.75) circle [radius=0.065];

\draw[fill=blue, draw=blue] (0,1) circle [radius=0.065];
\draw[fill=blue, draw=blue] (2,1.33) circle [radius=0.065];
\draw[fill=blue, draw=blue] (3,2) circle [radius=0.065];
\draw[fill=blue, draw=blue] (4,2.33) circle [radius=0.065];
\draw[fill=blue, draw=blue] (6,4) circle [radius=0.065];

\draw[red, thick] plot [variable=\t, domain=1.55:4.45] ({\t, \t*\t-6*\t+10.9});
\end{tikzpicture}
\caption{A bound and a partial bound of $y=\Phi_C(x)$.}
\label{fig:genfour_bound}
\end{figure}

For future reference we will need the following result.

\begin{lemma}[\cite{LN17}*{Lemma 4.2}] \label{lemma:genfour_slope3}
Let $C$ be a general curve of genus four, and let $E$ be a stable vector bundle of rank $r \geq 2$ and $\mu(E)=3$. Then $h^0(C, E)/r \leq 3/2$. 
\end{lemma}

\section{Two semi-orthogonal decompositions} \label{sec:sod}

The goal of this section is to describe the Kuznetsov component of a nodal cubic threefold, following the exposition in \cite{AK25}*{App. A}. We fix $Y \subset \P^4$ a nodal cubic threefold, and we keep the setup of Section \ref{sec:genfour}.

We recall that the \emph{Kuznetsov component} of $Y$ is defined by the semi-orthogonal decomposition 
\begin{equation} \label{eq:sod_Y}
\DC^b(Y) = \langle \Ku(Y), \O_Y(-1), \O_Y \rangle.
\end{equation}
The first step towards describing $\Ku(Y)$ is given by producing a semi-orthogonal decomposition on $X$, guided by \eqref{eq:sod_Y}.

\begin{prop}[cf. \cite{AK25}*{(A.2)}]
We have a semi-orthogonal decomposition 
\begin{equation} \label{eq:sod_X}
\DC^b(X) = \langle \widetilde{\Ku}(Y), \O_X(-H), \O_X, \O_Q \rangle.
\end{equation}
The pushforward $R\pi_\ast\colon \DC^b(X) \to \DC^b(Y)$ induces a Verdier quotient $\widetilde{\Ku}(Y) \to \Ku(Y)$, whose kernel is generated by the two objects $\O_Q(-1, 0), \O_Q(0, -1) \in \widetilde{\Ku}(Y)$.
\end{prop}

\begin{proof}
The first part is a direct computation. Note from \cite{KS23}*{Corollary 5.6 and Lemma 5.7} that $R\pi_\ast\colon \DC^b(X) \to \DC^b(Y)$ is a Verdier quotient, and the kernel is generated by the three objects $\O_Q(-1, -1), \O_Q(-1, 0), \O_Q(0, -1)$. The result follows immediately by the compatibility of the semi-orthogonal decompositions on $\DC^b(X)$ and $\DC^b(Y)$.
\end{proof}

\begin{remark} \label{remark:sod_crepant}
Note that each of the two objects $\O_Q(-1, 0), \O_Q(0, -1)$ are exceptional. This way, it is possible to refine \eqref{eq:sod_X} by writing
\[ \widetilde{\Ku}(Y) = \langle \O_Q(-1, 0), \D_1 \rangle = \langle \O_Q(0, -1), \D_2 \rangle. \]
In this case, the main result of \cite{AK25}*{App. A} ensures that each map $\D_i \to \Ku(Y)$ is a \emph{crepant} categorical resolution of singularities. We will go back to this point in Section \ref{sec:questions}
\end{remark}

A second semi-orthogonal decomposition on $X$ can be obtained using Orlov's semi-orthogonal decomposition for the blow-up of $C \subset \P^3$, namely
\[ \DC^b(X) = \langle i_\ast \beta^\ast \DC^b(C) \otimes \O_X(B), \O_X, \O_X(F), \O_X(2F), \O_X(3F) \rangle. \]
After some mutations, we get the following result.

\begin{teo}[\cite{AK25}*{Theorem A.1}]
Let $L_1, L_2$ be the two line bundles on $C$ induced by the two rulings of $Q \subset \P^3$. Then, we have an equivalence $\widetilde{\Ku}(Y) \cong \DC^b(\T_C)$, in such a way that the two objects $\O_Q(-1, 0), \O_Q(0, -1)$ are mapped to $\afrak(L_1), \afrak(L_2)$.
\end{teo}

\section{Proof of Theorem \ref*{teo:intro_stab}} \label{sec:pfmain}

In this section we will prove Theorem \ref{teo:intro_stab}. We will divide our exposition in two steps. In Step 1 we will construct the weak stability condition $\sigma_{3, 2}$ in $\DC^b(\T_C)$. Then, in Step 2 we will descend this to a stability condition $\overline{\sigma}_{3, 2}$ in $\Ku(Y)$.

\subsubsection*{Step 1: Constructing $\sigma_{3, 2}$}

Let us start by constructing $\sigma_{3, 2}$ as a weak pre-stability condition on $\DC^b(\T_C)$. By Theorem \ref{teo:stab_FN}, together with the explicit bound of Proposition \ref{prop:genfour_bound}, we have that $\sigma_{3, w} = (\A(3), Z_{3, w})$ defines a stability condition on $\DC^b(\T_C)$ for any $w>3$. We point out that the stable objects of phase 1 are $(\C, 0; 0)[1]$ and $(H^0(C, E), E; \ev_E)$ for any stable vector bundle $E \in \Coh(C)$ of slope $3$. 

This way, we set $\sigma_{3, 2} = (\A(3), Z_{3, 2})$. By continuity, we have that any $T \in \A(3)$ has $\Im Z_{3, 2}(T) \geq 0$, and if equality holds then $\Re Z_{3, 2}(T) \leq 0$.

\begin{claim}
Any object in $\A(3)$ admits a Harder--Narasimhan filtration by $\sigma_{3, 2}$-semistable objects. This follows by the same argument of \cite{FN25}*{pp. 13--4}; see \cite{Cho24}*{Proof of Proposition 6.7} for a discussion on why the argument still works for weak central charges. 
\end{claim}

Our strategy for the support property is based on reducing it to the support property of $\sigma_{3, w}$ for $w>2$. The first step is the following observation.

\begin{lemma} \label{lemma:pfmain_presupp}
Let $T \in \A(3)$ be a $\sigma_{3, 2}$-stable object with $\Im Z_{3, 2}(T)>0$. Then there exists some $\delta_0$ such that $E$ is $\sigma_{3, 2+\delta}$-stable for every $0<\delta<\delta_0$.
\end{lemma}

\begin{proof}
Note that $Z_{3, 2}$ has discrete image. This way, given $T \in \A(3)$ $\sigma_{3, 2}$-stable with positive slope, we have that there is an $\epsilon>0$ such that
\[ \frac{\nbf(T')-2\rbf(T')}{\dbf(T')-3\rbf(T')} \leq \frac{\nbf(T)-2\rbf(T)}{\dbf(T)-3\rbf(T)}-2\epsilon \]
for every $0 \neq T' \subsetneq T$. On the other hand, it is clear that for every $T' \subsetneq T$, we have $\rbf(T') \geq - \rbf(H^{-1}(T')) \geq -\rbf(H^{-1}(T))$. This way, for every $\delta<\epsilon/(\rbf(H^{-1}(T))+1)$ we have that
\[ \frac{\nbf(T')-(2+\delta)\rbf(T')}{\dbf(T')-3\rbf(T')} \leq \frac{\nbf(T)-2\rbf(T)}{\dbf(T)-3\rbf(T)}-\epsilon. \]
By continuity, this is smaller than the $Z_{3, 2+\delta}$-slope of $T$ for $0<\delta \ll 1$. 
\end{proof}

\begin{cor} \label{cor:pfmain_supp}
The weak pre-stability condition $\sigma_{3, 2}$ satisfies the support property with respect to $Q(r, d, n)=(d-3r)^2+r^2\cdot 1.9-nr$.
\end{cor}

\begin{proof}
We need to show that if $T \in \A(3)$ is a $\sigma_{3, 2}$-stable object with $Z_{3, 2}(T) \neq 0$, then $Q(\vbf(T)) \geq 0$. We divide into two cases.

First, assume that $\Im Z_{3, 2}(T)>0$. In this case, we proved in Lemma \ref{lemma:pfmain_presupp} that $T$ is $\sigma_{3, w}$-stable for some $w>2$; and $\Im Z_{3, w}(T)= \Im Z_{3, 2}(T)>0$. The result then follows by Lemma \ref{lemma:stab_quad} and Corollary \ref{cor:genfour_bound}.

Second, assume that $\Im Z_{3, 2}(T) =0$. In this case either $T=(\C, 0; 0)[1]$, or $T=(H^0(C, E), E; \ev_E)$ for some stable vector bundle $E$ of slope $\mu(E)=3$. The first case is clear. For the second one, note that either $\rbf(T)=1$ and $\nbf(T) \leq 2$ (as $Z_{3, 2}(E) \neq 0$), or $\rbf(T)\geq 2$. The first case is direct, and the second one follows from Lemma \ref{lemma:genfour_slope3}.
\end{proof}

\subsubsection*{Step 2: Descent to $\overline{\sigma}_{3, 2}$}

The last remaining part to conclude the proof is to show that $\sigma_{3, 2}$ descends to a stability condition on $\Ku(Y) = \DC^b(\T_C)/\K$, where $\K=\langle \afrak(L_1), \afrak(L_2) \rangle$ is the category generated by the BN-exceptional objects associated to the two trigonal line bundles $L_1, L_2$ on $C$.

First of all, let us note that the natural map $K(\T_C) \to \Z^3$ induced by the numerical Grothendieck group induces a map
\[ K(\DC^b(\T_C)/\K) \to \Z^3/\vbf(\afrak(L_i)) = \Z^3/(1, 3, 2)\cong \Z^2. \]
Moreover, it is clear that the central charge $Z_{3, 2}$ descends to a central charge $\overline{Z}_{3, 2}\colon \Z^3/(1, 3, 2) \to \C$, as $Z_{3, 2}(1, 3, 2)=0$.

The next step is to construct a heart on $\DC^b(\T_C)/\K$. To do so, the basic idea is to descend the heart $\A(3)$.

\begin{lemma} \label{lemma:pfmain_heart}
The category $\K \cap \A(3)$ is a Serre subcategory of $\A(3)$. Thus, the quotient $\A(3)/(\K \cap \A(3))$ is an abelian category, and the heart of a bounded t-structure on $\DC^b(\T_C)/\K$.
\end{lemma}

\begin{proof}
The key observation is that $\K= \{ T \in \DC^b(T_C): \lim_{w \to 2^+}m_{\sigma_{3, w}}(T) =0 \}$. This way, the argument of \cite{Bol23}*{Proposition 5.5} applies almost verbatim.
\end{proof}

This way, we define $\overline{\sigma}_{3, 2} \in \Stab(\DC^b(\T_C)/\K)$ via $\overline{\sigma}_{3, 2} = (\overline{Z}_{3, 2}, \A(3)/(\K \cap \A(3)))$. The support property is immediate as our lattice has rank two. This gives us a Bridgeland stability condition on $\Ku(Y)$.

\begin{remark}
Note that one could have use the main result of \cite{Bol23} directly to construct the Bridgeland stability condition on $\Ku(Y)$. In fact, the \emph{limiting support property} of \cite{Bol23}*{Definition 4.3} follows directly by the proof of Corollary \ref{cor:pfmain_supp}.
\end{remark}

\section{A moduli space} \label{sec:moduli}

In this section we will prove Corollary \ref{cor:intro_moduli}. Let $\vbf=\vbf((\C, \O_C(p+q); \id)[1])=(-1, -2, -1)$. For each $w>2$, we let $M_w=M_{\sigma_{3, w}}(\vbf)$ be the moduli space of $\sigma_{3, w}$-semistable objects with numerical class $\sigma_{3, w}$.

\begin{lemma} \label{lemma:moduli_stable}
For each $w>2$ and each $p, q \in C$, the object $(\C, \O_C(p+q); \id)[1] \in \A(3)$ is $\sigma_{3, w}$-stable. 
\end{lemma}

\begin{proof}
Assume that there exists a destabilizing sequence
\begin{equation} \label{eq:moduli_destab}
0 \to S \to (\C, \O_C(p+q); \id)[1] \to T \to 0
\end{equation}
in $\A(3)$. Here $T=H^{-1}(T)[1]$, and we have an exact sequence
\[ 0 \to H^{-1}(S) \to (\C, \O_C(p+q); \id) \to H^{-1}(T) \to H^0(S) \to 0 \]
in $\T_C$. Note that $j^!(H^{-1}(S)) = 0$ or $\O_C(p+q)$, as otherwise $j^!(H^{-1}(T))$ will have a torsion subsheaf.
\begin{enumerate}
\item If $j^!(H^{-1}(S))=\O_C(p+q)$, then $j^!(H^{-1}(T)) = j^!(H^0(S))=0$, by the construction of the torsion pair $\Ttors^3, \Ftors^3$. This forces $T=0$, and so \eqref{eq:moduli_destab} does not destabilize. 
\item If $j^!(H^{-1}(S))=0$, then we have a short exact sequence
\[ 0 \to \O_C(p+q) \to j^!(H^{-1}(T)) \to j^!(H^0(S)) \to 0. \]
Write $r=\rk j^!(H^0(S))$, $d=\deg j^!(H^0(S))$, so that $j^!(H^{-1}(T))$ has degree $d+2$ and rank $r+1$. This implies $(d+2)/(r+1)\leq 3 < d/r$ by the construction of the torsion pair, and so $d=3r+1$ by an elementary argument. But then $S$ has phase 1, and so \eqref{eq:moduli_destab} does not destabilize. \qedhere
\end{enumerate}
\end{proof}

\begin{cor}
For every $w>2$, there is an irreducible component of $M_w$ isomorphic to $\Sym^2(C)$
\end{cor}

\begin{proof}
Note that $\Sym^2(C) \to M_w$ defined via $p+q \mapsto (\C, \O_C(p+q); \id)[1]$ is well defined by Lemma \ref{lemma:moduli_stable}, and injective on closed points as $C$ is not hyperelliptic. A fast computation using Lemma \ref{lemma:aug_functors} shows that the moduli space $M_w$ is smooth of dimension 2 at each $(\C, \O_C(p+q); \id)[1]$ (as these objects are simple and with vanishing $\Ext^2$), and that the map at the level of tangent spaces from $\Sym^2(C)$ to $M_w$ is an isomorphism.     
\end{proof}

A natural question is to see what happens when $w=2$. We will not discuss the subtleties of defining the moduli space $M_2$ and $\overline{M}_2$ of objects on $\A(3)$, resp. $\overline{\A}(3)$. Instead, we will content ourselves with the following observation.

\begin{remark}
Let $r \in C$ be a point. Recall that $C \subset \P^3$ lies in a quadric surface; thus, there are two lines passing through $r$ contained in the quadric surface. In other words, we have isomorphisms $L_i \cong \O_C(p_i+q_i+r)$ for some $p_1+q_1, p_2+q_2 \in \Sym^2(C)$. This way, we get the short exact sequence
\[ 0 \to (\C, \O_r; \id) \to (\C, \O_C(p_i+q_i); \id)[1] \to \afrak(L_i)[1] \to 0 \]
in $\A(3)$. Thus, in $\Ku(Y)$ we get the isomorphisms
\[ (\C, \O_C(p_1+q_1); \id)[1] \cong (\C, \O_r; \id) \cong (\C, \O_C(p_2+q_2); \id)[1]. \]

It is easy to see that these are the only objects parametrized by $M_2$ that get identified in $\Ku(Y)$. In fact, assume that $T_j=(\C, \O_C(p_j+q_j); \ev)[1]$, $j=1,2$ are two objects in $\A(3)$, whose images are isomorphic in $\overline{\A(3)}$. Assume that there is no $r \in C$ and $i=1, 2$ such that $\O_C(p_j+q_j+r)\cong L_i$. We claim that $T_1 \cong T_2$ in $\A(3)$.

By Lemma \ref{lemma:pfmain_heart}, we get that there exists some $C \in \A(3)$ and morphisms $f_j\colon C \to T_j$ such that $\ker f_j, \coker f_j \in \K \cap \A(3)$. We immediately get that the $f_j$ are surjective, and that $\ker f_j$ are an iterated extension of $\afrak(L_i)[1]$. The fact that $\Hom(\afrak(L_i)[1], T_j)=0$ gives us that $\ker f_1 = \ker f_2$, from where the result follows. Small modifications apply for the objects $(\C, \O_r; \id)$.

We point out that the objects identified in $M_2 \to \overline{M}_2$ are exactly the identifications that relate $\Sym^2(C)$ with the Fano variety of lines of the associated nodal cubic threefold, cf. \citelist{\cite{vdGK10}*{p. 29} \cite{Huy23}*{Remark 5.5.2}}.
\end{remark}

\section{Questions} \label{sec:questions}

Let us finish our exposition with two future directions.
\begin{enumerate}
\item Recall that for a smooth cubic threefold $Y$ there are two ways of constructing a Bridgeland stability condition on $\Ku(Y)$: using the techniques of \cite{BLMS23}, or by relating $\Ku(Y)$ with the derived category of a conic fibration (cf. \cite{BMMS12}). The main result of \cite{FP23} ensures that these two stability conditions lie in the same $\tilde{\GL}_2^+$ orbit. 

It would be interesting to adapt the constructions above to a nodal cubic threefold, and to compare them with the stability condition from Theorem \ref{teo:intro_stab}. We point out that the results of \cite{FP23} do not apply directly to this situation, as the category $\Ku(Y)$ is not proper if $Y$ is singular.

\item As we mentioned in Remark \ref{remark:sod_crepant}, the main result of \cite{AK25}*{App. A} proves that the category $\DC^b(\T_C)$ admits two semi-orthogonal decompositions
\[ \DC^b(\T_C) = \langle \afrak(L_1), \K_1 \rangle = \langle \afrak(L_2), \K_2 \rangle; \]
and each $\K_i$ provides a \emph{crepant} categorical resolution of $\Ku(Y)$. It would be interesting to see whether the categories $\K_i$ admit stability conditions. Compare this to the introduction of \cite{Sun25}, where similar categories are expected to \emph{not} admit Bridgeland stability conditions.
\end{enumerate}

\bibliography{Cohgenfour}
\bibliographystyle{plain}

\end{document}